\documentclass{siamart190516}
\usepackage{amsfonts, amsmath}
\usepackage{algorithm, algpseudocode}

\usepackage[numbers,sort&compress]{natbib}
\usepackage{multirow}
\usepackage{tikz}
\usepackage{cleveref}
\usepackage{subcaption}
\usepackage{bibentry}
\usepackage[group-separator={,}]{siunitx}
\numberwithin{equation}{section}

\title{Robust Parameter Inversion Using Adaptive Reduced Order Models\thanks{This material is based upon work supported by the National Science Foundation under Grant No. 1720305. Any opinions, findings, and conclusions or recommendations expressed in this material are those of the author(s) and do not necessarily reflect the views of the National Science Foundation.}}
\author{Drayton Munster \and Eric de Sturler}

\newcommand{\R}{\mathbb{R}}

\newcommand{\C}{\mathbb{C}}

\newcommand{\pd}[2]{\dfrac{\partial #1}{\partial #2}}
\newcommand{\abs}[1]{\left|#1\right|}
\newcommand{\norm}[1]{\abs{\abs{#1}}}
\newcommand{\eps}{\epsilon}

\newcommand{\inv}[1]{\left( #1 \right) ^{-1}}
\newcommand{\invt}[1]{\left( #1 \right) ^{-T}}
\newcommand{\del}{\nabla}
\newcommand{\dx}[1][x]{\,d#1}
\DeclareMathOperator{\trace}{trace}
\DeclareMathOperator{\range}{Range}
\DeclareMathOperator{\kernel}{Kernel}
\DeclareMathOperator*{\argmin}{\arg\!\min}
\newcommand{\expect}[1]{\mathbb{E} \left[ #1 \right]}

\def\h{\eta}

\def\x{\chi}

\def\y{\psi}

\newcommand{\Rn}[1]{\mathbb{R}^{#1}}

\def\wt{\widetilde}

\newcommand{\M}[1]{ \mathbf{#1} }  

\newcommand{\V}[1]{ \mathbf{#1} }    
\newcommand{\VE}[2]{\MakeLowercase{#1}_{#2}} 

\newcommand{\Vgreek}[1]{ \boldsymbol{#1} }    
\newcommand{\Vgh}{\Vgreek{\h}}

\newcommand{\Ss}[1]{\mathcal{#1}}

\def\i{\mathrm{i}}
\newcommand{\spatialvar}{\V{x}}
\newcommand{\timevar}{t}
\newcommand{\states}{\V{y}}
\newcommand{\observation}{m}
\newcommand{\observables}{\V{\observation}}
\newcommand{\param}{p}
\newcommand{\params}{\V{\param}}
\newcommand{\control}{u}
\newcommand{\controls}{\V{\control}}

\newcommand{\numparam}{{n_{\mathrm{p}}}}
\newcommand{\numstate}{n}
\newcommand{\numcontrol}{{n_{\mathrm{src}}}}
\newcommand{\numobservable}{{n_{\mathrm{det}}}}
\newcommand{\numreduced}{{n_\mathrm{r}}}
\newcommand{\dotA}{\M{A}}
\newcommand{\dotE}{\M{E}}
\newcommand{\dotC}{\M{C}}
\newcommand{\dotB}{\M{B}}
\newcommand{\transfer}{\Psi}
\newcommand{\reduced}[1]{#1_{\mathrm{R}}}
\newcommand{\trueeval}[1][\params]{F\left( #1 \right)}
\newcommand{\romeval}[2][]{{\reduced{F}^{(#1)}} \left( #2 \right)}
\newcommand{\estimate}[1][\params]{F_{\mathrm{E}}\left( #1 \right)}
\newcommand{\lightspeed}{\nu}
\newcommand{\frequency}{\omega}
\newcommand{\diffusion}{D(\spatialvar)}
\newcommand{\absorption}{\mu(\spatialvar)}
\newcommand{\descconst}[1][\frequency]{\frac{\i{#1}}{\lightspeed}}
\newcommand{\descriptor}[1][\params]{\descconst \dotE-\dotA\left(#1\right)}

\newcommand{\descK}{\M{K}}
\newcommand{\paramdomain}{\R^{\numparam}}
\newcommand{\spatialdomain}{\Ss{X}}
\newcommand{\numfreq}{{n_\frequency}}
\newcommand{\photonflux}{\eta(\spatialvar,\timevar)}
\newcommand{\lightinput}{g(\spatialvar, \timevar)}
\newcommand{\inputspatial}{\VE{b}{j}(\spatialvar)}
\newcommand{\inputsrc}{\inputspatial\VE{\control}{j}(\timevar)}
\newcommand{\boundaryrefl}{a}
\newcommand{\fourier}[1]{\widehat{#1}}
\newcommand{\allobservations}{\M{M}}
\newcommand{\alldata}{\M{D}}
\newcommand{\projV}{\M{V}}
\newcommand{\projW}{\M{W}}
\newcommand{\reducedE}{\reduced{\dotE}}
\newcommand{\reducedA}{\reduced{\dotA}}
\newcommand{\reducedB}{\reduced{\dotB}}
\newcommand{\reducedC}{\reduced{\dotC}}
\newcommand{\reducedK}{\reduced{\descK}}
\newcommand{\reduceddescriptor}[1][\params]{\descconst \reducedE-\reducedA\left(#1\right)}
\newcommand{\residual}[1][\params]{\M{R}\left( #1 \right)}
\newcommand{\redresidual}[1][\params]{\reduced{\M{R}}\left( #1 \right)}
\newcommand{\probreject}{\alpha}
\newcommand{\probact}{\beta}

\newcommand{\ens}{\begin{enumerate}}
\newcommand{\ene}{\end{enumerate}}

\newcommand{\its}{\begin{itemize}}
\newcommand{\ite}{\end{itemize}}

\newcommand{\des}{\begin{description}}
\newcommand{\dee}{\end{description}}

\newcommand{\ars}[1]{\left[ \begin{array}{#1}}
\newcommand{\are}{\end{array} \right] }
\newcommand{\oars}[1]{\begin{array}{#1}}
\newcommand{\oare}{\end{array}}
\newcommand{\rars}[1]{\left( \begin{array}{#1}}
\newcommand{\rare}{\end{array} \right) }

\newcommand{\eqs}{\begin{eqnarray}}
\newcommand{\eqe}{\end{eqnarray}}
\newcommand{\eqsn}{\begin{eqnarray*}}
\newcommand{\eqen}{\end{eqnarray*}}

\newcommand{\myGlobalTransformation}[2]
{
    \pgftransformcm{1}{0}{0.4}{0.5}{\pgfpoint{#1cm}{#2cm}}
}

\begin{document}
\maketitle
\begin{abstract}
Nonlinear parametric inverse problems appear in many applications and are typically very expensive to solve, especially if they involve many measurements. These problems pose huge computational challenges as evaluating the objective function or misfit requires the solution of a large number of parameterized partial differential equations, typically one per source term. Newton-type algorithms, which may be required for fast convergence, typically require the additional solution of a large number of adjoint problems.

The use of parametric model reduction may substantially alleviate this problem. In 
[de Sturler, E., Gugercin, S., Kilmer, M. E., Chaturantabut, S., Beattie, C., \& O’Connell, M. (2015). Nonlinear Parametric Inversion Using Interpolatory Model Reduction. SIAM Journal on Scientific Computing, 37(3)], interpolatory model reduction was successfully used to drastically speed up inversion for Diffuse Optical Tomography (DOT). However, when using model reduction in high dimensional parameter spaces, obtaining error bounds in parameter space is typically intractable. In this paper, we propose to use stochastic estimates to remedy this problem. At the cost of one (randomized) full-scale linear solve per optimization step we obtain a robust algorithm. Moreover, since we can now update the model when needed, this robustness allows us to further reduce the order of the reduced order model and hence the cost of computing and using it,  further decreasing the cost of inversion.
We also propose a method to update the model reduction basis that reduces the number of large linear solves required by 46\%-98\% compared to the fixed reduced-order model.
We demonstrate that this leads to a highly efficient and robust inversion method.

\end{abstract}
\section{Introduction\label{sec:Intro}}

Nonlinear parameter inversion involves finding a set of parameters that minimizes the difference (or misfit) between the output of a parametric forward model and measured data. This minimization requires many evaluations of the forward model. When the forward model is described by discretized partial differential equations (PDEs), the number of large linear solves may be computationally intractable. To reduce this cost, the forward model can be approximated with a Reduced-Order Model (ROM). In \cite{Kragel2005}, the author examines the use of a ROM for the trust region sub-problem based on Proper Orthogonal Decomposition (POD). Since this still requires evaluation of the Full Order Model (FOM) to compute the improvement ratio, a multi-level strategy based on a hierarchy of successively finer spatial discretizations was used to defer the cost of the (highest-order) FOM until close to the optimum. However, this still requires many evaluations using the full-order (or nearly so) model. Using interpolatory, projection-based ROMs for Diffuse Optical Tomography (DOT) has been investigated in \cite{DeSturler2015}, yielding a very efficient method. However, these results do not provide an a posteriori error bound without evaluating the misfit using the FOM. 
Certain classes of problems (for example, coercive parabolic and elliptical PDEs \cite{Alla2017,Qian2016}) and ROMs admit a posteriori error bounds in parameter space. When these bounds are not available, a globally certified reduced basis can be constructed with the ``greedy POD'' method \cite{Bui-Thanh2008,Qian2016} in an ``offline-online'' approach, where a global basis is constructed beforehand (``offline'') and reused for particular problems (``online''). This ``offline'' construction is typically also computationally intractable for high-dimensional parameter spaces since the reduced basis must be sufficiently accurate across the entire parameter space. 
To reduce the costs of this ``offline'' phase and the costs of using a larger ROM, various adaptive strategies have been examined. In \cite{Afanasiev2010}, the authors use an adaptive POD basis to approximate an optimal control. This basis is iteratively updated by computing an optimal control for the ROM and adding the (full-order) state corresponding to that control to the snapshot basis. The iteration stops when subsequent controls differ by less than a given tolerance. The interpolation of POD bases is another offline strategy, as examined in \cite{Amsallem2008,Borggaard2014} (and references therein). In \cite{Peherstorfer2015}, the authors incorporate new data during the online phase by expanding the DEIM basis with low-rank updates chosen to reduce the approximation error at a randomly selected subset of components, and in \cite{DrusSimZas_2014}, for non-parametric systems, a dynamical approach for building the reduced model basis is considered based on residual norms. 

In this paper, instead, we investigate the use of stochastic techniques for trace estimation to efficiently estimate the error at the current step and provide a robust optimization technique. A significant reduction in the number of large linear solves is achieved by using a small initial projection basis and using these estimates to indicate when the bases should be expanded. To further reduce the number of large linear solves, we propose an update scheme inspired by \cite{OConnell2017} to minimize the number of additional vectors used to extend the projection basis. Our approach can also be considered an extension of the residual norm-based approach in
\cite{DrusSimZas_2014}.

The paper is organized as follows. \Cref{sec:Background} introduces the DOT problem, interpolatory model reduction, trust region methods, and trace estimation techniques. We introduce our proposed algorithm using stochastic estimates to guide ROM updates in \cref{sec:GuideUpdates}. \Cref{sec:BasisUpdates} discusses two possible choices for efficiently updating the projection basis to improve the accuracy of the ROM. While this paper focuses on efficiency, \cref{sec:Propabilistic_Estimates} outlines the number of samples necessary for probabilistic guarantees of the accuracy of these estimates for problems that require more robustness. We explore some of the choices with numerical experiments and their performance implications in \cref{sec:Numerical_Experiments}. Finally, we provide conclusions and discuss future work in \cref{sec:Conclusions}.

\section{Background\label{sec:Background}}

We first describe the forward model for DOT and express this problem in the system theoretic notation used in the remainder of the paper. This section borrows heavily from \cite{DeSturler2015}. Next, we give background on the construction of reduced frequency response functions via interpolatory model reduction. The third subsection provides background on estimating the trace of a matrix and its application to estimating Frobenius norms.

\subsection{The DOT Problem\label{subsec:Background_DOT}}

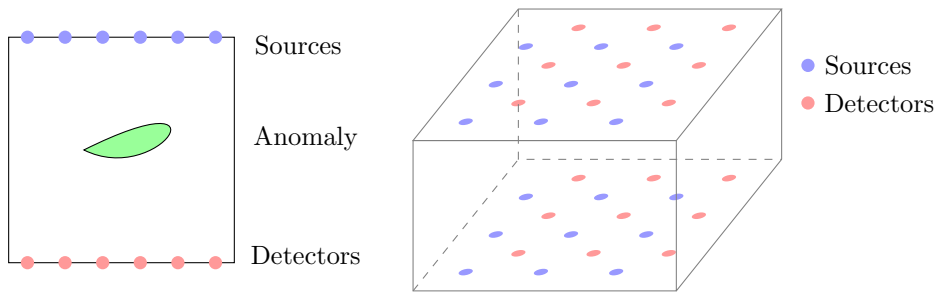
\begin{figure}[t]
\centering
\begin{minipage}{.4\textwidth}
\begin{tikzpicture}
\draw (0,0) rectangle (3,3);
\foreach \x in {.25, .75, 1.25, 1.75, 2.25, 2.75}
{
	\fill[blue!40!white] (\x,3) circle (.25em);
	\fill[red!40!white]  (\x,0) circle (.25em);
}
\fill[green!40!white, draw=black] (1,1.5) .. controls (3,2.5) and (2,1) .. cycle;
\node[align=left] at (3.85,2.9) {Sources};
\node[align=left] at (3.95,1.65) {Anomaly};
\node[align=left] at (3.95,0.1) {Detectors};
\end{tikzpicture}
\end{minipage}
\begin{minipage}{.55\textwidth}
\centering
\begin{tikzpicture}

\begin{scope}
\myGlobalTransformation{0}{0};

\foreach \x in {0.5, 1.5, 2.5} {
	\foreach \y in {0.5,1.5,2.5} {
		\fill[blue!40!white] (\x, \y) circle (.25em);
		\fill[red!40!white] (\x+.5, \y+.5) circle (.25em);	
	}
}
\end{scope}

\begin{scope}
\myGlobalTransformation{0}{2};
\draw [black!50] (0,0) rectangle (3.5,3.5);

\foreach \x in {0.5, 1.5, 2.5} {
	\foreach \y in {0.5,1.5,2.5} {
		\fill[blue!40!white] (\x, \y) circle (.25em);
		\fill[red!40!white] (\x+.5, \y+.5) circle (.25em);	
	}
}
\end{scope}
\draw [black!50] (0,0) rectangle (3.5,2) 
(3.5,0) -- (4.9,1.75)
(4.9,1.75) -- (4.9,3.75);
\draw [black!50,dashed] (0,0) -- (1.4,1.75)
(1.4,3.75) -- (1.4,1.75) -- (4.9,1.75);
\fill[blue!40] (5.25,3) circle (.25em);
\draw node at (6.05, 3) {Sources};
\fill[red!40] (5.25,2.5) circle (.25em);
\draw node at (6.2, 2.5) {Detectors};
\end{tikzpicture}
\end{minipage}
\caption{2D and 3D Geometries}
\label{fig:geometries}
\end{figure}

We use a diffusion model \cite{Arridge1999} for the photon flux, $\photonflux$, driven by an input light source $\lightinput$. 
In practice, the DOT problem is posed in the frequency domain. Here, we pose the problem in the time domain to motivate the use of parametric inversion techniques from a system theoretic perspective. Light is transmitted from one of $\numcontrol$ physically stationary sources, so we write $\lightinput = \inputsrc$ for the source locations $\inputspatial$, $j=1, \ldots, \numcontrol$. Observations are made using an array of $\numobservable$ detectors located along the boundary. See \cref{fig:geometries} for examples used in \cref{sec:Numerical_Experiments}. Let $\VE{\observation}{i}(\timevar)$ denote the observed flux at detector $i$ and time $\timevar$. With the above notation, we model the diffusion and absorption of light by the following time-dependent partial differential equation,
\begin{align}\label{eq:DOT_PDE1}
\frac{1}{\lightspeed} \pd{}{t}\photonflux &= \del \cdot \left( \diffusion\del\photonflux \right) - \absorption\photonflux + \inputsrc, \text{ for } \spatialvar \in \spatialdomain\\
0 &= \photonflux + 2 \boundaryrefl \diffusion \pd{}{\xi}\photonflux, \text{ for } \spatialvar \in \partial\spatialdomain \setminus \partial\spatialdomain_{\pm} \label{eq:DOT_PDE2}\\
0 &= \photonflux, \text{ for } \spatialvar \in \partial\spatialdomain_{\pm}\label{eq:DOT_PDE3}\\
\VE{\observation}{i}(\timevar) &= \int_{\partial \spatialdomain} \VE{c}{i}(\spatialvar)\photonflux \dx[\spatialvar]\label{eq:DOT_PDE4}.
\end{align}
Here, $\spatialvar$ refers to a spatial location in our image domain $\spatialdomain$, $\partial\spatialdomain_{\pm}$ refer to the top and bottom of the image domain (respectively), $\boundaryrefl$ is a constant defining the diffusive boundary reflection, $\diffusion$ and $\absorption$ refer to diffusion and absorption fields (respectively), $\xi$ refers to the outward unit normal on the boundary, and $\lightspeed$ is the speed of light in the medium.

The inverse problem associated with DOT is to reconstruct $\diffusion$ and $\absorption$, the diffusion and absorption fields (respectively), given a set of observations $\observables(t)$ made by illuminating the domain with a variety of source signals $\controls(\timevar)$. In this paper, we assume that the diffusion field $\diffusion$ is known (a common assumption in DOT breast tissue imaging). This leaves the reconstruction of the absorption field, $\absorption$. We represent this field with a finite number of (unknown) parameters, $\params = [\param_1, \ldots, \param_\numparam]^T$.
The choice of parameterization is crucial for physiologically relevant solutions to the inverse problem. In this setting, $\absorption$ is well-approximated by a piecewise constant, two-valued function. As for many inverse problems, the naive approach is ill-posed and extremely sensitive to noise in the measurements. To efficiently describe complex geometries with sharp boundaries and provide regularization, we use parametric level sets (PaLS), developed in \cite{Aghasi2011a} and used in the context of DOT imaging in \cite{Aghasi2011a,DeSturler2015}.

The spatial discretization of \cref{eq:DOT_PDE1,eq:DOT_PDE2,eq:DOT_PDE3,eq:DOT_PDE4} yields the system of differential algebraic equations,
\begin{align}
\frac{1}{\lightspeed}\dotE\dot{\states}(\timevar; \params) &= -\dotA(\params) \states(\timevar; \params) + \dotB\controls(\timevar), \label{eq:DOT_DAE_1}\\
\observables(\timevar; \params) &= \dotC^T \states(\timevar; \params), \label{eq:DOT_DAE_2}
\end{align}
where $\states$ denotes the discretized photon flux, $\observables = \left[ \observation_1, \ldots, \observation_\numobservable\right]^T$ is the vector of detector outputs, $\dotC^T\states$ approximates \cref{eq:DOT_PDE4} via quadrature, the columns of $\dotB$ represent discretizations of the sources $\inputspatial$, $\dotA(\params) = \dotA_0 + \dotA_1(\params)$, where $\dotA_0$ and $\dotA_1$ are discretizations of the diffusion and absorption terms, respectively. $\dotE$ includes the discretization of boundary terms such as the Robin condition \cref{eq:DOT_PDE2} and is singular as a result.
Let $\fourier{\states}(\frequency; \params),\fourier{\observables}(\frequency; \params), \fourier{\controls}(\frequency; \params)$ denote the Fourier transform of $\states(\timevar; \params),\observables(\timevar; \params)$, and $\controls(\timevar; \params)$, respectively. Taking the Fourier transform of \cref{eq:DOT_DAE_1,eq:DOT_DAE_2} yields
\begin{align}
\fourier{\observables}(\frequency; \params) &= \transfer(\frequency; \params) \fourier{\controls}(\frequency; \params), \text{\; where}\\
\transfer(\frequency; \params) &= \dotC^T \inv{\descK(\frequency; \params)} \dotB, \label{eq:transfer}
\end{align}
where $\frequency \in \R$ and $\descK(\frequency; \params) = \descriptor$. $\transfer(\frequency; \params)$ is the \emph{frequency response function} of the dynamical system \cref{eq:DOT_DAE_1}.

For a given frequency $\frequency_j$ and input source $i$, we denote the predicted observations by the forward model $\fourier{\observables}_i(\frequency_j; \params)$. For a given parameter vector $\params$, the predicted observations for all $\numcontrol$ input sources and $\numfreq$ frequencies are given by
\begin{equation}
\allobservations(\params) = \left[ \fourier{\observables}_1(\frequency_1; \params), \ldots, \fourier{\observables}_\numcontrol(\frequency_1; \params), \fourier{\observables}_1(\frequency_2; \params), \ldots, \fourier{\observables}_\numcontrol(\frequency_\numfreq; \params) \right],
\end{equation}
where $\allobservations(\params) \in \C^{\numobservable \times \numcontrol \cdot \numfreq}$. 
In our case, $\fourier{\controls}_i(\frequency_j) = \V{e}_i$, representing the excitation of source $i$ with a pure frequency $\frequency_j$. The evaluation of $\allobservations(\params)$ then reduces to the evaluation of the frequency response function for each frequency 
\begin{equation} \label{eq:fun_eval}
\allobservations(\params) = \left[\transfer(\frequency_1; \params), \ldots, \transfer(\frequency_\numfreq; \params)\right].
\end{equation}
Given the empirical data matrix of observations, $\alldata$, the optimization problem that must be solved is
\begin{equation}\label{eq:minimization}
\min_{\params \in \paramdomain} \norm{\residual}_F = \min_{\params \in \paramdomain} \norm{\allobservations(\params) - \alldata}_F.
\end{equation}

We assume $\alldata$ contains additive noise at a known noise level. Since PALS regularizes the problem \cite{Aghasi2011a}, no further regularization is necessary and we terminate the optimization when $\norm{\residual}$ is at the noise level (or slightly above).

Although the number of systems is independent of the number of parameters, each objective function evaluation requires $\min(\numobservable, \numcontrol)\cdot \numfreq$ large linear solves.
Since we use Newton-type methods to minimize \cref{eq:minimization}, it is also necessary to compute the Jacobian of the objective function. Differentiation of $\transfer(\frequency_k; \params)$ with respect to parameter $p_\ell$ yields
\begin{equation}
\label{eq:jacobian}
\frac{\partial}{\partial \param_\ell} \transfer(\frequency_k; \params) = -\dotC^T \inv{\descK(\frequency_k; \params)} \left(\frac{\partial}{\partial \param_\ell} \dotA(\params)\right) \inv{\descK(\frequency_k; \params)} \dotB.
\end{equation}
Since $\frac{\partial}{\partial \param_\ell} \dotA(\params)$ is diagonal and inexpensive to compute in our case, the bulk of the computational cost is in solving the systems $\descK(\frequency_k; \params)\M{X} = \dotB$ and $\left(\descK(\frequency_k; \params)\right)^T \M{Y} = \dotC$. Note that one of these solutions is already available from the objective function evaluation.

\subsection{Interpolatory Model Reduction\label{subsec:Background_ModelRed}}

Since solving \cref{eq:minimization} is dominated by the cost of computing \cref{eq:fun_eval} and \cref{eq:jacobian}, we build a surrogate frequency response function, $\reduced{\transfer}(\frequency; \params)$, that maintains a high-fidelity approximation to $\transfer(\frequency; \params)$ using techniques from projection based parametric model reduction \cite{Benner2015, DeSturler2015}.

Assuming the state, $\states$, evolves near the $\numreduced$-dimensional subspace $\range(\projV)$ for some appropriately chosen $\projV \in \C^{\numstate \times \numreduced}$, i.e.  $\states(\timevar; \params) \approx \projV\reduced{\states}(\timevar; \params)$ and given a full-rank $\projW \in \C^{\numstate \times \numreduced}$, we enforce the Petrov-Galerkin condition
$$ \projW^T \left( \frac{1}{\lightspeed} \dotE \projV \reduced{\dot{\states}}(\timevar; \params) + \dotA(\params) \projV \reduced{\states}(\timevar; \params) - \dotB \controls(\timevar) \right) = 0, \; \observables(\timevar; \params) = \dotC^T \projV \reduced{\states}(\timevar; \params)  $$
to yield the reduced system
\begin{equation}
\frac{1}{\lightspeed}\reduced{\dotE}\reduced{\dot{\states}}(\timevar; \params) = -\reduced{\dotA}(\params) \reduced{\states}(\timevar; \params) + \reduced{\dotB}\controls(\timevar), \label{eq:Reduced_DOT_DAE_1} \; \reduced{\observables}(\timevar; \params) = \reduced{\dotC}^T \reduced{\states}(\timevar; \params)
\end{equation}
and associated reduced frequency response function
\begin{equation}
\label{eq:Reduced_Transfer}
\reduced{\transfer}(\frequency; \params) = \reducedC^T \inv{\reduceddescriptor} \reducedB,
\end{equation}
where $\reducedE = \projW^T \dotE \projV$, $\reducedA = \projW^T \dotA \projV$, $\reducedB = \projW^T \dotB$, and $\reducedC = \projV^T \dotC$. We similarly define $\reducedK(\frequency; \params) = \reduceddescriptor$.


For an accurate ROM, we must choose appropriate $\projV$ and $\projW$ to satisfy desired fidelity requirements. A natural choice comes from the following result in interpolatory model reduction \cite{Baur2011},
\begin{theorem}
\label{thm:interpolation}
Suppose $\dotA(\params)$ is continuously differentiable in a neighborhood of $\params_0 \in \R^{\numparam}$. Let $\frequency \in \R$ and suppose both $\descK(\frequency; \params_0)$ and $\reducedK(\frequency; \params_0)$ be invertible. If
$\inv{\descK(\frequency; \params_0)}\dotB \subset \range\left(\projV\right)$ and
$\invt{\descK(\frequency; \params_0)}\dotC \subset \range\left(\projW\right)$, then the reduced model, $\reduced{\transfer}$, satisfies
\begin{align*}
\transfer(\frequency;\params_0) &= \reduced{\transfer}(\frequency; \params_0), \\
\del_\frequency\transfer(\frequency;\params_0) &= \del_\frequency\reduced{\transfer}(\frequency; \params_0), \text{\; and}\\
\del_\params\transfer(\frequency;\params_0) &= \del_\params\reduced{\transfer}(\frequency; \params_0).  
\end{align*}
\end{theorem}
So, both function and Jacobian evaluations exactly match at any point where these conditions are met. This gives us a recipe for constructing the projection and test spaces. For a set of interpolation points, $\pi_1, \ldots, \pi_k$, we construct $\projV$ from a (numerically stable) basis for $$\range\left(\left[\inv{\descK(\frequency_1; \pi_1)}\dotB, \ldots, \inv{\descK(\frequency_1; \pi_k)}\dotB,  \ldots, \inv{\descK(\frequency_{\numfreq}; \pi_k)}\dotB\right]\right)$$ and we construct $\projW$ from a (numerically stable) basis for $$\range\left(\left[\invt{\descK(\frequency_1; \pi_1)}\dotC, \ldots, \invt{\descK(\frequency_1; \pi_k)}\dotC,  \ldots, \invt{\descK(\frequency_\numfreq; \pi_k)}\dotC\right]\right).$$

In general, where these conditions are not satisfied, there is a difference between the reduced-order and the full-order model's function evaluation. The following theorem, a minor extension from \cite[Theorem 3.1]{Beattie2012}, relates the error in the reduced frequency response function to linear system residuals for any vector in $\range(\dotB)$ (the notation has been modified to fit this paper). This theorem motivates a very efficient approach to control the error between reduced order model and full order model, discussed in \cref{subsec:Updates_Residual}. 

Given two subspaces of $\C^n$, say $\Ss{M}$ and $\Ss{N}$, we can define the (sine of the) angle\footnote{Note that this definition is not symmetric unless $\dim\Ss{M} = \dim\Ss{N}$.} between these subspaces, $\Theta(\Ss{M}, \Ss{N}) \in [0, \frac{\pi}{2}]$, as $$\sup_{\V{x} \in \Ss{M}} \inf_{\V{y} \in \Ss{N}} \frac{\norm{\V{y} - \V{x}}_2}{\norm{\V{x}}_2} = \sin \Theta(\Ss{M}, \Ss{N}).$$ 
\begin{theorem}
\label{thm:forward_error}
Given the full order frequency response function, $\transfer(\frequency, \params)$ and reduced order frequency response function,
$\reduced{\transfer}(\frequency, \params)$, the
tangential interpolation error at $(\frequency, \params)$ in the (arbitrary)
direction $\V{z} \in \Rn{ n_\mathrm{src} }$ is given by
\begin{equation}
  \norm{\reduced{\transfer}(\frequency, \params)\V{z} - \transfer(\frequency, \params)\V{z}}_2
  \leq
  \norm{\dotC^T \inv{\descK(\frequency; \params)}}_2 
  \frac{ \sin \Theta\left( \Ss{C}(\frequency,\params), \Ss{W} \right) }{ \cos \Theta\left( \widehat{\Ss{B}}(\frequency,\params), \Ss{W} \right) }
  \norm{\Vgh}_2,
\end{equation}
where\footnotemark
\begin{align*}
  \Vgh & = \descK(\frequency; \params) \projV \V{t} -  \dotB\V{z} \quad \mbox{with} \\
  \V{t} & = \argmin_{ \wt{\V{t}} \in \Rn{n_\mathrm{r}} }
      \left\| \descK(\frequency; \params) \projV \wt{\V{t}} - \dotB\V{z}  \right\|_2 ,
\end{align*}
\footnotetext{Equivalently, $\Vgh$ is the smallest residual for $\descK(\frequency; \params) \V{x} = \dotB\V{z}$ with the constraint $\V{x} \in \range(\projV)$.}
$\widehat{\Ss{B}}(\frequency,\params) = \range \left(\descK(\frequency; \params) \projV \right)$,
$\Ss{C}(\frequency,\params) = \left(\kernel\left( \dotC^T \inv{\descK(\frequency; \params)}\right) \right)^\perp$, and
$\Ss{W} = \range \left(\projW\right)$.
\end{theorem}

\begin{proof}
Note that $\left(\descriptor\right) \projV \V{t} = \dotB \V{z} + \Vgh \in \widehat{\Ss{B}}(\frequency, \params)$. The rest of the proof follows identically from the proof of \cite[Theorem 3.1]{Beattie2012}.
\end{proof}

In practice, it is not efficient to compute $\norm{\dotC^T \inv{\descK(\frequency; \params)}}_2$  or the angle between $\mathcal{C}(\frequency, \params)$ and $\Ss{W}$ since these are equivalent in work to a full-order function evaluation. However, \cref{thm:forward_error} shows that the relative error in the function evaluation is bounded by the relative residuals norms (up to an unknown constant). 
Hence, to reduce the interpolation error at a point to a given tolerance at minimum cost, that is, with the smallest possible basis extension, we should include the directions with largest tangential interpolation error. We discuss this further in \cref{subsec:Updates_Residual}.

\subsection{Trust region methods} \label{subsec:Background_TR}
Robust Newton-type methods use a local model of the objective function to compute a candidate
update. This is most commonly done using a line search or a trust region.  Here, we focus on trust region methods \cite{Conn2000} and a local model based on the Gauss-Newton approximation.
Let $F(\params) = \frac{1}{2} \norm{\residual}_F^2$ be the objective function, where
$\residual$ is the nonlinear residual from ($\ref{eq:minimization}$).
Furthermore, let $\M{J}(\params)$ be the Jacobian of the nonlinear residual.
Then, given a current approximation $\params_c$, corresponding residual
$\V{r}_c$, and Jacobian $\M{J}_c$, the local (Gauss-Newton) model is given by
\begin{equation}
\label{eq:TR_locmod}
  m_\mathrm{GN}(\params)  =   \frac{1}{2}\left(\V{r}_c^T \V{r}_c)
    + \V{r}_c^T \M{J}_c\right(\params - \params_c) + \frac{1}{2} (\params - \params_c) \M{J}_c^T\M{J}_c (\params - \params_c) .
\end{equation}
To compute a candidate update, this local model is (approximately) minimized over a neighborhood of $\params_c$. We use the regularized trust region method TREGS \cite{DeSturler2011} which uses a truncated SVD combined with a GCV-like condition to compute a
candidate update $\V{s} = \params -\params_c$.
Trust-region methods subsequently accept the candidate update
if the improvement of the objective function is larger
than a chosen sufficient improvement predicted by the trust region model, in our case,
\begin{equation} \label{eq:ImprovCheck}
  F(\params_c) - F(\params_c+\V{s})  \geq 
    \rho_k  ( m_\mathrm{GN}(\params_c) - m_\mathrm{GN}(\params_c+\V{s}) ),
\end{equation}
for a chosen parameter $\rho_k$ \cite{Conn2000}.
To avoid expensive function and Jacobian evaluations, we replace the objective function and its derivatives in the Gauss-Newton model by approximations using an interpolatory ROM \cite{DeSturler2015, Borcea2014}.
This drastically reduces the computational cost \cite{DeSturler2015} but introduces a number of complications.

First, in general, our parametric ROM does not satisfy the usual requirement in trust-region methods that the Gauss-Newton model and its gradient are exact at the current parameter point. This problem can be remedied by using so-called conditional models \cite[Chapter 9]{Conn2000}\cite{Yue2013}. Using conditional models, convergence can be proved if (1) the first order accuracy at the current point can be restored (exact or to sufficient accuracy) in a finite number of model improvements, and (2) for a sufficiently small trust region size, steps that satisfy \cref{eq:ImprovCheck} can be guaranteed.
Condition (2) is a standard condition for trust-region methods and is satisfied by TREGS. Condition (1) can be satisfied in a single step by adding the current point as an interpolation point to the ROM.
While this is always possible, adding an interpolation point is quite expensive, 
as it requires the solution of a large linear system for
each source and frequency combination and an adjoint (transpose) solve for
each detector and frequency combination. 

However, to determine the need for conditional model improvement requires either (1) a full function evaluation or (2) a global error bound on the ROM approximation \cite{Conn2000,Yue2013}. For interpolatory
model reduction in a high dimensional parameter space, the cost of providing
an error bound over the entire parameter region of interest is, in general, prohibitive \cite{Bui-Thanh2008}.
We propose to use, instead, randomized estimates of the objective function to decide when to improve the ROM; we discuss this further
in \cref{sec:GuideUpdates,sec:BasisUpdates}.


\subsection{Trace Estimation\label{subsec:Background_TraceEst}}

Recall that for $A \in \R^{m \times n}$,
$$\trace(\M{A}^T \M{A}) = \norm{\M{A}}_F^2. $$
Therefore, we can estimate the Frobenius norm of $\M{A}$ by estimating the trace of $\M{A}^T\M{A}$.

Let $\V{s} \in \R^{n}$ be a vector with independent, identically distributed entries with mean zero and unit variance. Then, we have
\begin{align*}
\expect{\norm{\M{A}\V{s}}_2^2} &= \expect{\V{s}^T \M{A}^T \M{A}\V{s}} =\sum_{i=1}^{n} \sum_{j=1}^{n} (\M{A}^T \M{A})_{ij} \expect{\VE{s}{i} \VE{s}{j}}\\
 &= \sum_{i=1}^{n} (\M{A}^T \M{A})_{ii} \expect{\VE{s}{i}^2} + \sum_{\substack{i=1,j=1 \\i \neq j}}^{n} (\M{A}^T \M{A})_{ij} \expect{\VE{s}{i} \VE{s}{j}}.
\end{align*}
Since $\VE{s}{i}$ and $\VE{s}{j}$ are independent and have mean zero, $\expect{\VE{s}{i} \VE{s}{j}} = \expect{\VE{s}{i}} \expect{\VE{s}{j}} = 0$. Since $\VE{s}{i}$ also has unit variance, $\expect{\VE{s}{i}^2} = 1$. Thus,
$$
\expect{\norm{\M{A}\V{s}}_2^2} = \trace(\M{A}^T\M{A}) = \norm{\M{A}}_F^2.$$
By selecting $\VE{s}{i}$ from the Rademacher distribution $\left[P(\VE{s}{i} = x) = \begin{cases}1/2 & x=1\\1/2 & x = -1\end{cases}\right]$, we obtain the Hutchinson trace estimator which minimizes the variance of the sample \cite{Hutchinson1990}.

\section{Estimating ROM Accuracy and Guiding Updates\label{sec:GuideUpdates}}

Since the objective function \cref{eq:minimization} involves the Frobenius norm, we can use the Hutchinson trace estimator to estimate the objective function.
This offers significant computational savings over the full function evaluation. Consider the evaluation of the following estimate. Given a sample $\V{s} \in \R^{\numcontrol}$ drawn from the Rademacher distribution,
\begin{align*}
\left(\transfer(\frequency_j; \params) - \alldata_j \right)\V{s} &= \left(\dotC^T \inv{\descK(\frequency_j; \params)}\dotB - \alldata_j\right)\V{s}\\
&= \dotC^T \inv{\descK(\frequency_j; \params)}\left(\dotB \V{s}\right) - \alldata_j \V{s}.
\end{align*}
Whereas the (full-order) objective function evaluation requires solving $\min(\numobservable, \numcontrol)$ large linear systems per frequency, this trace estimate requires the solution of only 1 large system per sample and per frequency.
While other methods use this estimate directly in the optimization \cite{Haber2012,Aslan2017}, we instead use it only to assess the quality of our reduced order model. In particular, we only need the estimate to indicate whether there is a significant difference between the reduced-order objective function and the full-order objective function evaluation.

Equipped with an efficient method to estimate the difference between the full-order and reduced-order objective function evaluations, we combine the ROM with our optimization routine, and we update the ROM if our estimates suggest a significant difference. For an iteration index $k$, let $\romeval[k]{\params}$ denote the objective function computed using the reduced order model generated from $\projV^{(k)}$ and $\projW^{(k)}$, $\estimate$ denote the trace estimate of $\trueeval$, and $d(\romeval[k]{\params}, \estimate)$ denote our metric for model quality. Then given a starting point $\params_0$, initial projection bases $\projV^{(0)},\projW^{(0)}$, optimizer tolerance tol$_\mathrm{o}$, model quality tolerance tol$_\mathrm{q}$, and an update strategy (discussed in \cref{sec:BasisUpdates}), we may apply \Cref{alg:rom-optim}.

\begin{algorithm}[ht]
\begin{algorithmic}
\State $\params_\mathrm{c} \gets \params_0, f_\mathrm{c} \gets \romeval[0]{\params_\mathrm{c}}, k \gets 0$
\While{$f_\mathrm{c} \geq \text{tol}_\mathrm{o}$}
\State $\params_\mathrm{p} \gets $ Compute Trust Region Candidate Point Using $\reduced{F}^{(k)}$
\State $f_\mathrm{p} \gets \romeval[k]{\params_\mathrm{p}}$
\State $f_{\mathrm{est}} \gets \estimate[\params_\mathrm{p}]$
\If{$d(f_\mathrm{p},f_{\mathrm{est}}) \geq \text{tol}_\mathrm{q}$}
\State Reject $\params_p$
\State $\projV^{(k+1)}, \projW^{(k+1)} \gets \text{update}(\projV^{(k)}, \projW^{(k)}, \params_\mathrm{c}, \params_\mathrm{p})$
\State $f_\mathrm{c} \gets \romeval[k+1]{\params_c}$
\State $k \gets k+1$
\Else
\If{Trust Region Algorithm Accepts $\params_\mathrm{p}$}
\State $\params_\mathrm{c} \gets \params_\mathrm{p}$
\State $f_\mathrm{c} \gets f_\mathrm{p}$
\EndIf
\EndIf
\EndWhile
\end{algorithmic}
\caption{ROM-Based Optimization with Estimate-Driven Updates}
\label{alg:rom-optim}
\end{algorithm}

In this paper: tol$_\mathrm{o}$ is set to 1.1 times the noise level (based on the discrepancy principle), TREGS \cite{DeSturler2011} is used as the optimization algorithm, and the reduced order model is rejected when $\estimate[\params]/\romeval[k]{\params} \geq 10$ at the proposed point.

 \section{Basis Updates}\label{sec:BasisUpdates}
When $d(f_\mathrm{p}, f_{\mathrm{est}}) \geq \text{tol}_\mathrm{q}$, it is necessary to update the projection bases. Below, we examine two possibilities: (1) using the interpolatory conditions from \cref{thm:interpolation} to make the reduced-order model exact at either the proposed or current point;  (2) coupling the residual-based error bounds from \cref{thm:forward_error} with techniques from subspace recycling \cite{OConnell2017, Ahuja2012, Ahuja2015} to add directions corresponding to large residuals (in the sense of \Cref{thm:forward_error}). We discuss performance comparisons in \cref{sec:Numerical_Experiments}.

\subsection{Full Interpolatory Updates\label{subsec:Updates_Full}}

When the objective function estimate suggests that the reduced and full order models differ significantly, we can use \cref{thm:interpolation} to make the reduced order model match the full order model at that point.  This has the cost of evaluating $\inv{\descK(\frequency; \params)}\dotB$ and $\left(\descK(\frequency; \params)\right)^{-T}\dotC$ for all frequencies $\omega$. In terms of large system solves, this at least the cost of 2 (full-order) function evaluations. In exchange, the reduced order model will match the full order model to first order at the given point. If we use the current point to update, this guarantees that the optimizer will choose the same step direction as the full order model.

As discussed in \cref{subsec:Background_TR}, the conditional model framework can be used to \emph{suggest} convergence. While interpolatory updates can be used to make the reduced-order model exact in a single step, the error estimates available through \cref{subsec:Background_TraceEst} are only probabilistic. We believe the ``relaxed first-order condition'' from \cite{Yue2013} can be extended to a probabilistic setting, and we will explore this in a future paper.

While \cref{thm:interpolation} suggests adding (in general) $\numfreq \cdot \left(\numobservable + \numcontrol\right)$ vectors when an updated is needed, our experiments suggest that these full updates are too expensive and contain a significant amount of redundant information. For example, in our 3D experiments (\cref{subsec:Experiements_3D}), \cref{thm:interpolation} implies that $1\cdot(961 + 961) = 1922$ vectors (and thus large linear solves) are sufficient for each update. 
However, at the first rejection point, only 170 basis vectors are added after a rank revealing factorization, while the remainder are (numerically) linearly dependent. This indicates that there is a much smaller bases expansion that will produce a reduced order model of essentially the same accuracy.

\subsection{Residual Driven Updates\label{subsec:Updates_Residual}}

To reduce the cost of updating the model, we turn to \cref{thm:forward_error} for motivation. This theorem indicates that reducing the norm of the residual in a given direction will reduce the interpolation error in that direction. More directly, we bound the interpolation error in $\residual$ as follows.

\begin{theorem}
\label{thm:residual_bound}
Let $\M{T}_j = \descK(\frequency_j; \params) \projV \inv{\reducedK(\frequency_j; \params)} \projW^T$, i.e., the projection onto $\range\left(\descK(\frequency_j; \params) \projV \right)$ along $\projW$, and define $\Vgh_j = \left(\M{I} - \M{T}_j\right)\dotB$. Define $\M{H} = [\Vgh_1, \ldots, \Vgh_{\numfreq}]$. Then \begin{align*}
  \abs{\,\norm{\residual}_F - \norm{\redresidual}_F\,} \leq \kappa(\params) \norm{\M{H}}_F,
\end{align*}
where $\kappa(\params)$ does not depend on $\projV, \projW$.
\end{theorem}

\begin{proof}
\begin{align*}
\abs{\,\norm{\residual}_F - \norm{\redresidual}_F\,}^2 &= \abs{\,\norm{\allobservations(\params) - \alldata}_F - \norm{\reduced{\allobservations}(\params) - \alldata}_F\,}^2\\
&\leq \norm{\allobservations(\params) - \reduced{\allobservations}(\params)}_F^2\\
&=\sum_{j=1}^{\numfreq} \norm{\dotC^T \left[ \inv{\M{K}(\frequency_j; \params)} - \projV \inv{\reduced{\M{K}}(\frequency_j; \params)} \projW^T \right] \dotB}_F^2\\
&= \sum_{j=1}^{\numfreq}\norm{\dotC^T \inv{\M{K}(\frequency_j; \params)}\left[ \M{I} - \M{T}_j \right] \dotB}_F^2 = 
\end{align*}
\begin{align*}
&= \sum_{j=1}^{\numfreq}\norm{\dotC^T \inv{\M{K}(\frequency_j; \params)} \Vgh_j}_F^2 \\
&\leq \sum_{j=1}^{\numfreq}\norm{\dotC^T \inv{\M{K}(\frequency_j; \params)}}_F^2 \norm{\Vgh_j}_F^2\\
&\leq \left(\max_{j=1,\ldots,\numfreq} \norm{\dotC^T \inv{\M{K}(\frequency_j; \params)}}_F\right)^2 \sum_{j=1}^{\numfreq} \norm{\Vgh_j}_F^2\\
&= \left(\max_{j=1,\ldots,\numfreq} \norm{\dotC^T \inv{\M{K}(\frequency_j; \params)}}_F\right)^2 \norm{\M{H}}_F^2 \\
&= \left(\kappa(\params) \norm{\M{H}}_F\right)^2.
\end{align*}
Taking square roots yields the desired bound.
\end{proof}

Since evaluating the condition number, $\kappa(\params)$, is equivalent in cost to evaluating the full order model (and thus undesirable), we use the objective function estimates from \cref{sec:GuideUpdates} to estimate the  necessary reduction in $\norm{\M{H}}_F$. Since $\kappa(\params)$ does not depend on $\projV$ or $\projW$, we are free to expand these projection bases to reduce the residual norm, $\norm{\M{H}}_F$.

For each $\frequency_j$, we compute the (thin) QR factorization $\descK(\frequency_j; \params) \projV = \M{Q}_j \widehat{\M{R}}_j$. Thus, $\M{Q}_j\M{Q}_j^T$ is the orthogonal projector onto $\range\left(\descK(\frequency_j; \params)\projV\right)$. From this, we have (dropping the $\frequency_j$ and $\params$ dependence for brevity):
\begin{align*}
\Vgh &= \M{Q}\M{Q}^T \Vgh + (\M{I} - \M{Q}\M{Q}^T) \Vgh\\
&= \M{Q}\M{Q}^T \left( \M{I} - \M{T} \right) \dotB +  (\M{I} - \M{Q}\M{Q}^T)\left( \M{I} - \M{T} \right) \dotB\\
&= \underbrace{\left( \M{Q}\M{Q}^T - \M{T} \right) \dotB}_{\subset \range\left(\descK\M{V}\right)} + \underbrace{\left(\M{I} - \M{Q}\M{Q}^T\right) \dotB}_{\subset \range\left(\descK\M{V}\right)^\perp}.
\end{align*}
The first term in the above decomposition represents the difference between the skew projector $\descK\projV \inv{\reducedK} \projW^T$ and the orthogonal projector $\descK\projV \M{\widehat{R}}^{-1}\M{Q}^T = \M{Q}\M{Q}^T$. Since this term is already in $\range \left(\descK \projV \right)$, we focus on $\left(\M{I} - \M{Q}\M{Q}^T\right) \dotB$ and in particular, we add vectors to $\projV$ from $\descK^{-1}\range\left(\left(\M{I} - \M{Q}\M{Q}^T\right) \dotB\right)$.

To minimize the number of basis vectors added (and thus the number of large linear solves), we use the singular value decomposition of $\left(\M{I} - \M{Q}\M{Q}^T\right) \dotB$ to determine the most significant components. Let $\M{U}\M{\Sigma}\M{Y}^T = \left(\M{I} - \M{Q}\M{Q}^T\right) \dotB$ be a singular value decomposition. We expand $\projV$ with $\left[ \descK^{-1}\V{u}_1, \ldots, \descK^{-1}\V{u}_r \right]$ for $r \ll \numcontrol$ that sufficiently reduces $\norm{(\M{I}-\M{Q}\M{Q}^T)\dotB}$. The Eckart-Young Theorem\cite{Eckart1936} shows that $\sum_{i=1}^{r} \sigma_i \V{u}_i \V{y}_i^T$ is the optimal rank-$r$ approximation in both the Frobenius and spectral norms. The relative error in this low-rank approximation in the Frobenius norm is given by $\sqrt{\dfrac{\sum_{i=r+1}^{\numcontrol} \sigma_i^2}{\sum_{i=1}^{\numcontrol} \sigma_i^2}}$. If we choose $r$ such that this relative error is less than some $\eps$ (e.g. $\eps = \frac{1}{10}, \frac{1}{20}, \frac{1}{100}$, etc.) and compute the QR factorization $\M{K}\left[\projV, \M{K}^{-1}\V{u}_1, \ldots, \M{K}^{-1}\V{u}_r\right] = \widetilde{\M{Q}} \widetilde{\M{R}}$, then it follows that $$\norm{\left[\M{I} - \widetilde{\M{Q}}\widetilde{\M{Q}}^T\right] \dotB}_F \leq \eps \norm{\left[\M{I} - \M{Q}\M{Q}^T\right] \dotB}_F \leq \eps \norm{\Vgh}_F.$$
We expect\footnote{Note that this is not guaranteed. The error term corresponding to the skew projection may increase in norm. However, the second term dominates the residual norm in our experiments. An analysis of how much the skew projection term can deteriorate is future work.} to see a similar factor $\eps$ improvement in $\norm{\Vgh_j}_F$ and hence in the interpolation error between the reduced-order and full-order models. In our experience, this is the case. This procedure is repeated for each $\frequency_i$ to compute the updates to $\projV$. 
This approach extends the residual norm-based approach in \cite{DrusSimZas_2014} in two important respects. First, it only considers the magnitude of the right hand side component in 
$\range(\M{K} \M{V})^{\perp}$, that is, the component that should be solved for (note that here all right hand sides have the same norm). Second, our approach does not consider residuals for specific right hand sides, but all possible residuals for vectors in $\range(\M{B})$. 

\begin{theorem}
\label{thm:residual_bound2}
Let $\Vgreek{\xi}_j = \left[\M{I} - \projV \inv{\reducedK(\frequency_j; \params)}\projW^T \descK(\frequency_j; \params)\right]^T \dotC$ and define $\M{\Xi} = [\Vgreek{\xi}_1, \ldots, \Vgreek{\xi}_{\numfreq}]$. Then $$\abs{\norm{\residual}_F - \norm{\redresidual}_F} \leq \mu(\params) \norm{\M{\Xi}}_F,$$ where 
$\mu(\params)$ does not depend on $\projV, \projW$.
\end{theorem}

The proof is nearly identical to that of \cref{thm:residual_bound} and a similar decomposition in terms of $\range(\descK^T \projW)$ and its orthogonal complement yield an update scheme for $\projW$.

\subsubsection{Example}

To illustrate the usefulness of this approach, we consider the first rejected step in the 2D 1-point ROM experiment in \cref{sec:Numerical_Experiments}.
As \cref{fig:update_resids} shows, the relative residual norm for each of the 32 columns of $\dotB$ (blue circles) is $O(10^{-4})$. However, the relative error in the objective function, $\frac{\norm{\romeval[0]{\params_p} - \trueeval[\params_p]}}{\norm{\trueeval[\params_p]}}$, is approximately 11.7\%. 
The singular values of $\left(\M{I} - \M{Q}\M{Q}^T\right) \dotB$ decay rapidly as shown in \cref{fig:updates_svs}. Therefore, the residual norm can be sufficiently decreased with the addition of very few singular vectors. As seen in \cref{fig:updates_reduction}, only two linear solves are needed for a 90\% reduction in relative error, 3 linear solves for a 95\% reduction, and 5 linear solves for a 99\% reduction.

With the addition of the three singular vectors corresponding to the largest three singular values of $\left(\M{I} - \M{Q}\M{Q}^T\right) \dotB$ (requiring three large linear solves), the relative residual norms are significantly reduced (orange circles in \cref{fig:update_resids}) and the relative error in the objective function is reduced to approximately $0.2\%$. 

\begin{figure}[ht]
\centering
\begin{subfigure}[b]{0.4\textwidth}
\includegraphics[width=\textwidth]{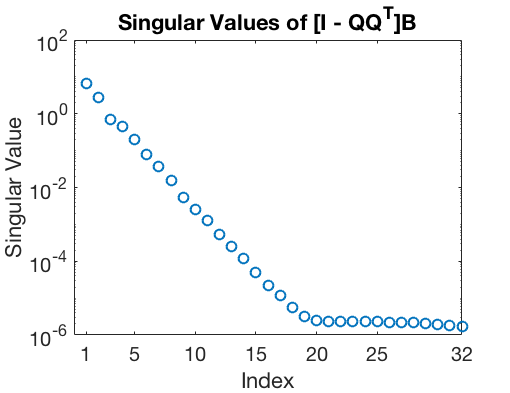}
\caption{Rapid Decay of Singular Values}
\label{fig:updates_svs}
\end{subfigure}
\begin{subfigure}[b]{0.4\textwidth}
\includegraphics[width=\textwidth]{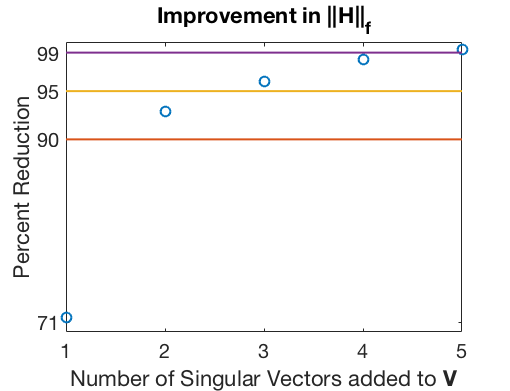}
\caption{Improvement in $\norm{\M{H}}_F$}
\label{fig:updates_reduction}
\end{subfigure}

\caption{Improvement in Error Bound at First Rejected Step in 2D - 1 Point Experiment}
\label{fig:updates_improvement}
\end{figure}

\begin{figure}[ht]
\includegraphics[width=\textwidth]{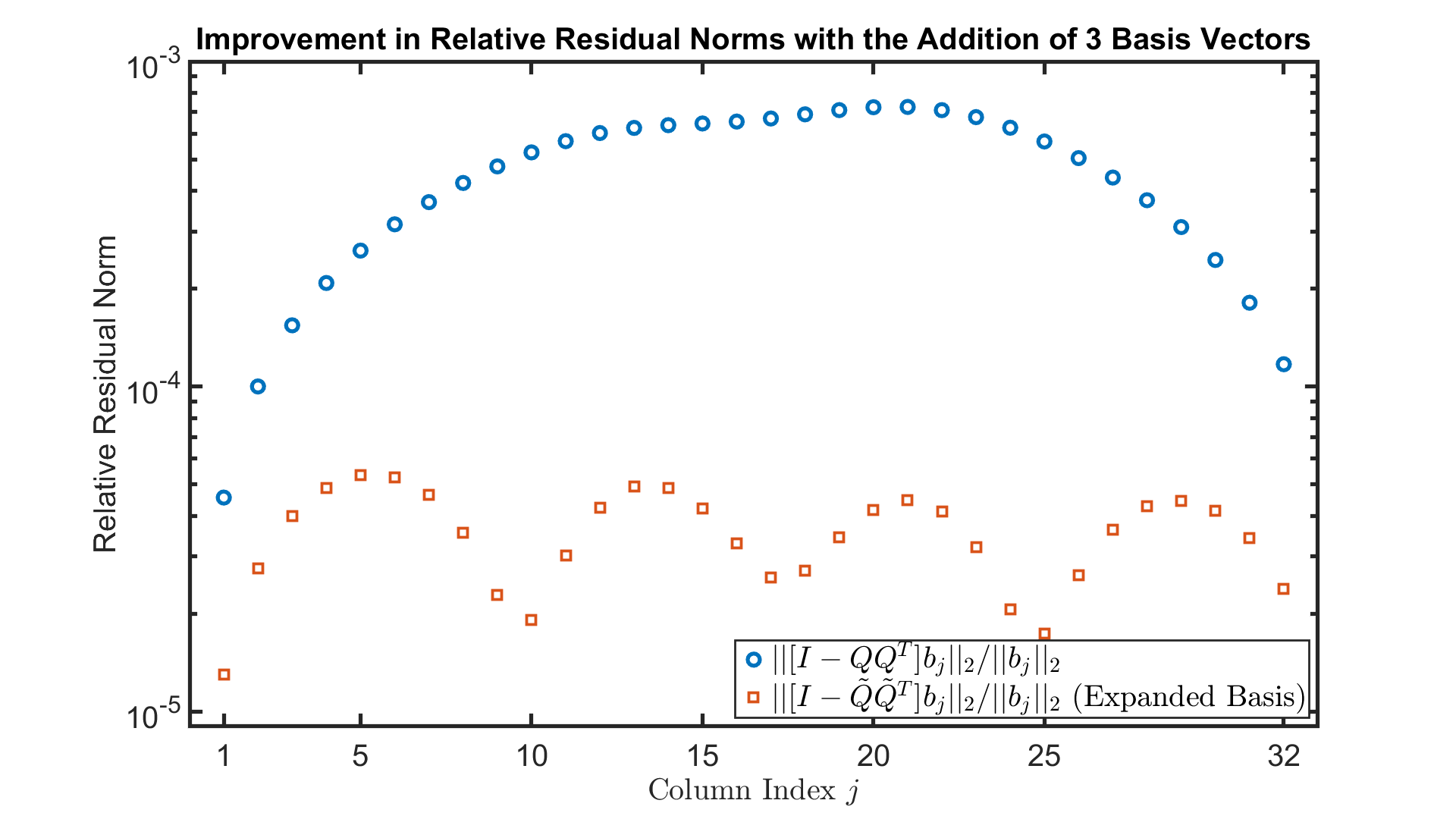}
\caption{Improvement in Relative Residual Norms}
\label{fig:update_resids}
\end{figure}
%
%


\section{Probabilistic Estimates\label{sec:Propabilistic_Estimates}}

Since each sample of the stochastic estimates of \cref{sec:GuideUpdates} requires the solution of one large linear system per frequency, selecting an appropriate number of samples is key to the efficiency and robustness of our proposed technique. A small number of samples increases the variance of the estimate which could trigger unnecessary updates or fail to detect a large difference between the reduced- and full-order objective functions. A large number of samples makes the algorithm more robust but requires solving more large linear systems. A detailed discussion on this efficiency-robustness trade-off is beyond the scope of this paper, but using the analysis outlined in \cite{Avron2011} and improved upon in \cite{Roosta-Khorasani2015}, we provide rigorous bounds on the number of samples required to detect a sufficiently large under-estimate of the objective function with a given probability.


We focus our attention on under-estimates of the objective function because a sufficiently large over-estimate will simply trigger an update to the model. While this would impact the efficiency of our method, over-estimates do not degrade the robustness. In contrast, successive under-estimates of the objective function could prevent updates and degrade our algorithm to traditional ROM-based parameter inversion. Therefore, we want to explore the number of samples required to make large under-estimates sufficiently unlikely.

To bound the number of samples required, we use an intermediate result from \cite{Roosta-Khorasani2015} which says that given $N \geq 6 \epsilon^{-2} \log(1 / \delta)$ samples of the Hutchinson estimator, denoted $\trace_{H}^N$, $P\big(\trace_{H}^N(\M{A})  \leq (1-\epsilon)\trace(\M{A})\big) < \delta$. That is, the probability of under-estimating the trace of $\M{A}$ by a relative factor of $\epsilon$ is less than $\delta$. 

Suppose we want to reject the reduced order model when $\romeval[k]{\params} > \probreject \trueeval[\params]$ for some $\alpha > 1$ with high probability and our rejection threshold is set to $\frac{\probreject}{K}$ for some $K > 1$, that is, we reject the reduced order model if $\estimate[\params] > \frac{\probreject}{K} \romeval[k]{\params}$. Let $\beta = \dfrac{\trueeval[\params]}{\romeval[k]{\params}}$. Then our model is rejected if 
\begin{align*}
    \estimate &\geq \frac{\probreject}{K} \romeval[k]{\params}\\
    &= \frac{\probreject}{K \probact} \trueeval\\
    &= \left(1 - \left(1 -  \frac{\probreject}{K \probact}\right) \right)\trueeval.
\end{align*}
This will occur with probability at least $1-\delta$ given $N \geq  6 \epsilon^{-2} \log(1 / \delta)$ samples where $\eps = \left(1 -  \frac{\probreject}{K \probact}\right)$. While $\probact$ is not typically available to us without a full function evaluation, we are interested in detecting cases where $\probact \geq \probreject$. In such cases, the minimum $\eps$ occurs at $\probreject = \probact$ with $\eps = \left(1 - \frac{1}{K} \right)$.

By choosing $K$ sufficiently large, we may detect such cases while requiring a modest number of samples. As an example, if $K=2$, then only $\lceil \frac{6}{\left(1 - \frac{1}{2}\right)^2} \log(2) \rceil = 17$ samples are required to guarantee a rejection probability, $\delta$, of at least 1/2. Increasing the number of samples to 34 yields a rejection probability of at least 3/4. Thus, by reducing the rejection threshold by a factor $\frac{1}{K}$ to provide a buffer and using a modest number of samples, we obtain strong lower bounds on the rejection probabilities.

\section{Numerical Experiments\label{sec:Numerical_Experiments}}

To demonstrate the value of our proposed techniques, we present a series of numerical experiments. To quantify the cost of each method, the experiments are evaluated by the number of large (i.e., the size of the full-order model) linear solves needed to drive the residual norm to 1.1 times the noise level. While there are computational costs for the update procedures (such as an SVD for the updates of \cref{subsec:Updates_Residual}), these involve linear algebra on systems multiple orders of magnitude smaller than the full-order model (e.g., $O(10^3)$ vs $O(10^5)$). Therefore, the dominant computational cost will be the large linear solves. Since the full-order systems require iterative solvers to be efficiently solved, the computational cost is (roughly) linear in the number of large solves.

We examine reconstructions for a two-dimensional and a three-dimensional tomography problem. For each, we compare the cost of reconstruction using the full-order model, using a reduced order model constructed by interpolating the first three full-order optimization steps (henceforth labeled 3-point), and using a reduced order model interpolating only the initial condition (henceforth labeled 1-point). 
For the reduced order models, we explore the accuracy and cost without updates, and we compare the resulting efficiency of the (exact) interpolatory updates of \cref{subsec:Updates_Full} with the residual-driven updates of \cref{subsec:Updates_Residual}.

These 1-point models are of particular interest because the same initial condition is (typically) used for each reconstruction. Thus, the solutions of $\inv{\descK(\frequency; \params)}\dotB$ and $\left(\descK(\frequency; \params)\right)^{-T}\dotC$ from this initial condition can be computed off-line and reused for any reconstruction based on the same mesh, source/detector location, and frequencies. This \emph{significantly} reduces the number of large-scale solves, especially for large problems. The number of large linear solves required beyond the off-line solves is henceforth referred to as the amortized cost.

The numerical experiments are constructed as follows. Synthetic data is generated by constructing a target anomaly\footnote{These anomalies are not exactly reconstructable in the PaLS basis, avoiding the ``inverse crime'' of using the same model for data generation and reconstruction.} in the 0-1 pixel basis. Background pixels are mapped to a background absorption level, $\mu_{out}$, while pixels belonging to the anomaly are mapped to a higher absorption value, $\mu_{in}$. A small inhomogeneity is added to the anomaly and the background and then measurements are then generated by evaluating the forward problem and adding white noise.

For these experiments, we use a symmetric, one-sided projection, i.e. $\projV = \projW$, and a single frequency, corresponding to $\frequency = 0$.


\subsection{2D Example}\label{subsec:Experiements_2D}

\newcommand{\twodseed}{20}
\newcommand{\noiseleveltwod}{1}

The 2D problem is evaluated on a $201 \times 201$ mesh, yielding $\numstate=40,401$ degrees of freedom for the large system. The standard centered finite difference scheme is used to discretize \cref{eq:DOT_PDE1,eq:DOT_PDE2,eq:DOT_PDE3,eq:DOT_PDE4}. We use 32 sources and 32 detectors arranged as shown in \cref{fig:geometries}. The true image and initial condition are shown in \cref{fig:2d_truth}.  We use 25 compactly-supported radial basis functions (CSRBF) for our PaLS reconstruction, yielding 100 (four per CSRBF) parameters to be estimated. The absorption values, $\mu_{out}$ and $\mu_{in}$, are set to 0.005 and 0.15, respectively. White noise is added to the measurements at a level of {\noiseleveltwod} permille. To quantify the impact of the random nature of the trace estimates of \cref{sec:Propabilistic_Estimates}, each experiment is repeated 103 times\footnote{A number of the form $4k+3$ was chosen to ensure that the median, 25th, and 75th percentiles were elements of the set.} and the minimum, 25th, 50th, 75th percentiles, and maximum number of large linear solves are reported in \cref{table:2D_results}. \Cref{fig:2d_reconstruction} shows the reconstructions for a representative noise realization. \Cref{fig:2D_convergence} details the convergence history of the 3-point ROM without updates and the 3-point ROM with the (exact) interpolatory updates of \cref{subsec:Updates_Full}.

\begin{figure}[ht]
\begin{center}
\includegraphics[width=\textwidth]{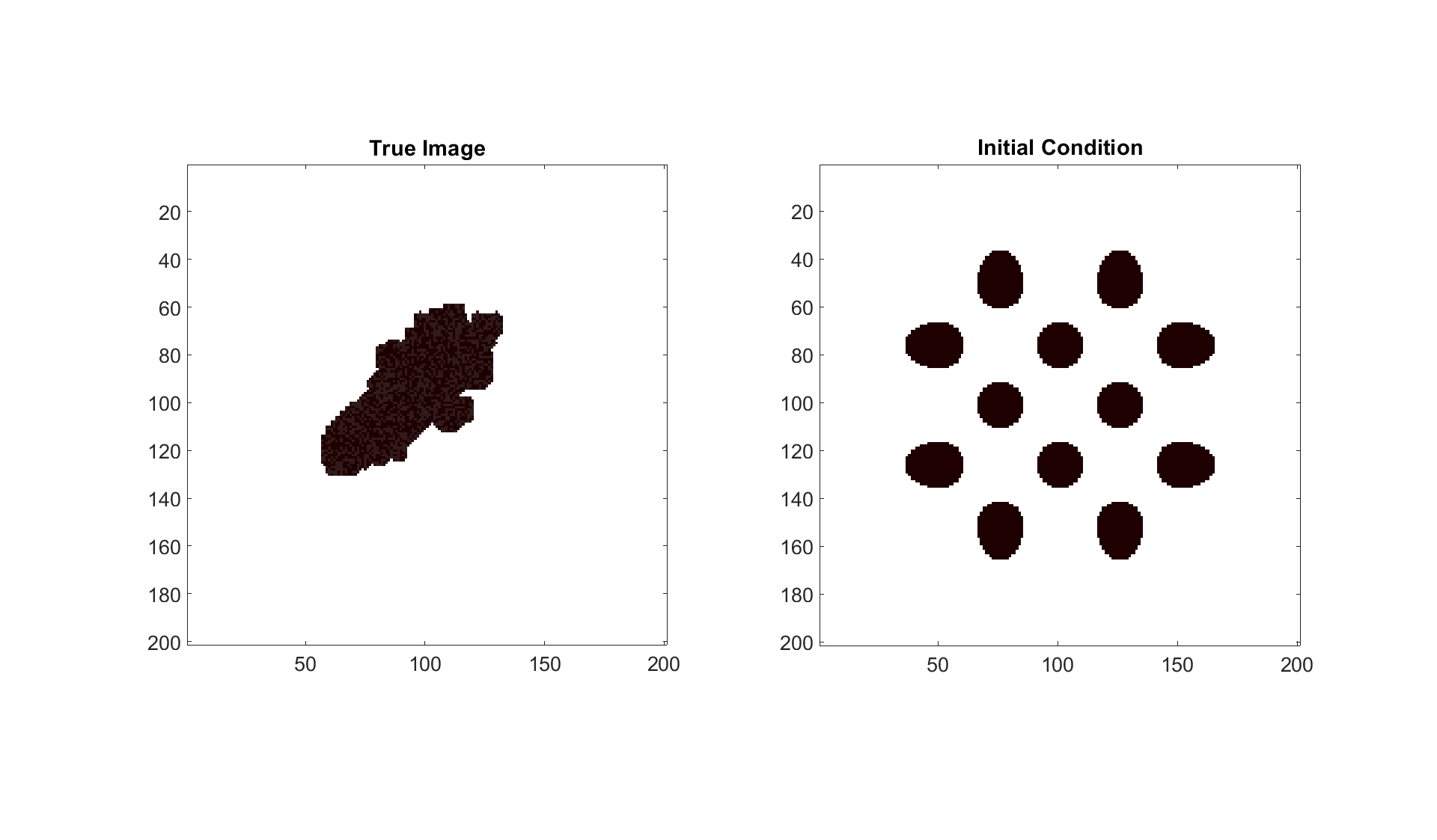}
\end{center}
\caption{True Image for the 2D Model and Initial Condition}
\label{fig:2d_truth}
\end{figure}

\subsubsection{Choice of Update Location}

As noted in \cref{sec:BasisUpdates}, there is a choice of updating using the currently accepted point, $\params_\mathrm{c}$, the proposed step $\params_\mathrm{p}$, or both if the current step is rejected. 

After the model is updated, \cref{alg:rom-optim} returns to $\params_\mathrm{c}$ to compute a new function evaluation and Jacobian. If interpolatory updates (\cref{subsec:Updates_Full}) are used to update at $\params_\mathrm{c}$, then this function evaluation and Jacobian will be exact (identical to those from the full-order model) according to \cref{thm:interpolation}. This guarantees that the trust region search direction for the reduced-order model is identical to the direction chosen for the full-order model at that point.

However, \cref{table:2D_results} demonstrates that using updates at $\params_\mathrm{c}$, in all but one setup, requires more large solves in the median minimization than using updates at $\params_\mathrm{p}$. We conjecture that this occurs for the following reason. In order for $\params_\mathrm{c}$ to have been accepted previously, the reduced-order model evaluated there, $\romeval[k]{\params_\mathrm{c}}$, is likely to be sufficiently close to the estimate $\estimate[\params_\mathrm{c}]$. Similarly, for an update to occur, the reduced-order model at the proposed point, $\romeval[k]{\params_\mathrm{p}}$, and the estimate, $\estimate[\params_\mathrm{p}]$, are likely to be sufficiently different. Thus, we expect that the error between the full-order and reduced-order models is larger at $\params_\mathrm{p}$ than at $\params_\mathrm{c}$. This suggests that updating at $\params_\mathrm{p}$ will yield more improvement in the reduced order model.

Since updating the reduced-order model at $\params_\mathrm{p}$ required an equal number or fewer large linear solves in the majority of our experiments, with nearly identical quality results, we use updates at $\params_\mathrm{p}$.


\subsubsection{Results}

\begin{table}[ht]%
\begin{tabular}{|l|c|c|c|c|c|}%
\hline%
\multirow{2}{*}{Experiment Setup}&Median Solves&\multicolumn{4}{c|}{Percentiles}\\%
\cline{3%
-%
6}%
&{[}Amortized{]}&Min&25&75&Max\\%
\hline%
\multirow{2}{*}{Full Order Model}&992&704&896&1120&2240\\%
&{[}928{]}&{[}640{]}&{[}832{]}&{[}1056{]}&{[}2176{]}\\%
\hline%
\multirow{2}{*}{ROM - 3 Point - Interpolatory Updates at $\V{x}_\mathrm{c}$}&279&207&274&283&373\\%
&{[}215{]}&{[}143{]}&{[}210{]}&{[}219{]}&{[}309{]}\\%
\hline%
\multirow{2}{*}{ROM - 3 Point - Interpolatory Updates at $\V{x}_\mathrm{p}$}&279&207&274&283&360\\%
&{[}215{]}&{[}143{]}&{[}210{]}&{[}219{]}&{[}296{]}\\%
\hline%
\multirow{2}{*}{ROM - 3 Point - No Updates}&192&192&192&192&192\\%
&{[}128{]}&{[}128{]}&{[}128{]}&{[}128{]}&{[}128{]}\\%
\hline%
\multirow{2}{*}{ROM - 3 Point - Residual-Driven Updates at $\V{x}_\mathrm{c}$}&219&207&215&223&245\\%
&{[}155{]}&{[}143{]}&{[}151{]}&{[}159{]}&{[}181{]}\\%
\hline%
\multirow{2}{*}{ROM - 3 Point - Residual-Driven Updates at $\V{x}_\mathrm{p}$}&218&207&215&223&244\\%
&{[}154{]}&{[}143{]}&{[}151{]}&{[}159{]}&{[}180{]}\\%
\hline%
\multirow{2}{*}{ROM - 1 Point - Interpolatory Updates at $\V{x}_\mathrm{c}$}&302&213&229&308&444\\%
&{[}238{]}&{[}149{]}&{[}165{]}&{[}244{]}&{[}380{]}\\%
\hline%
\multirow{2}{*}{ROM - 1 Point - Interpolatory Updates at $\V{x}_\mathrm{p}$}&240&214&230&302&311\\%
&{[}176{]}&{[}150{]}&{[}166{]}&{[}238{]}&{[}247{]}\\%
\hline%
\multirow{2}{*}{ROM - 1 Point - Residual-Driven Updates at $\V{x}_\mathrm{c}$}&138&107&133&140&155\\%
&{[}74{]}&{[}43{]}&{[}69{]}&{[}76{]}&{[}91{]}\\%
\hline%
\multirow{2}{*}{ROM - 1 Point - Residual-Driven Updates at $\V{x}_\mathrm{p}$}&132&101&125&135&145\\%
&{[}68{]}&{[}37{]}&{[}61{]}&{[}71{]}&{[}81{]}\\%
\hline%
\end{tabular}%
\caption{Number of Large Linear Solves for the 2D Problem}%
\label{table:2D_results}%
\end{table}

\begin{figure}
\centering
\begin{subfigure}[b]{0.8\textwidth}
\includegraphics[width=\textwidth]{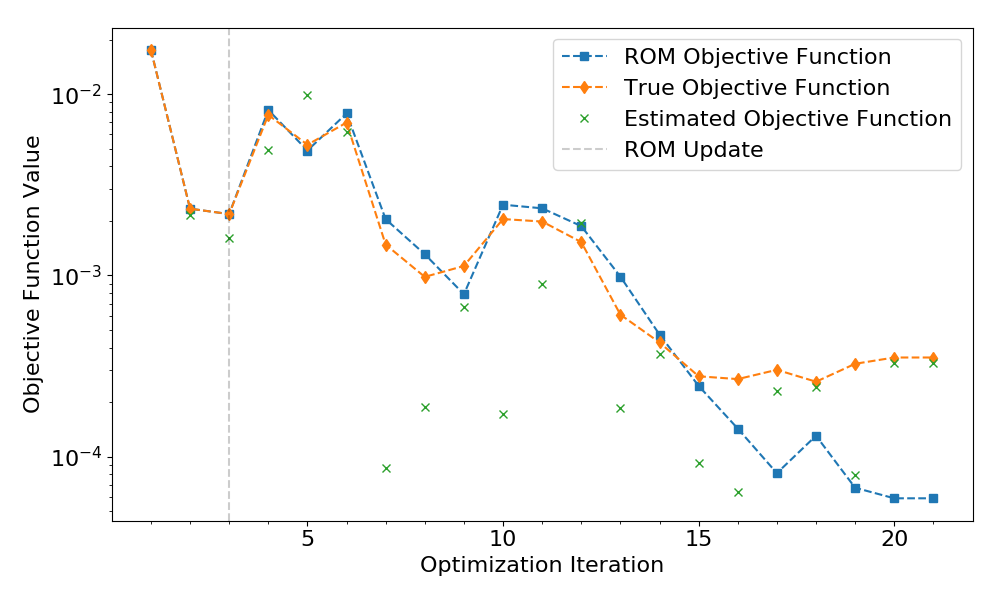}
\caption{No Updates}
\label{fig:2D_convergence_none}
\end{subfigure}
\begin{subfigure}[b]{0.8\textwidth}
\includegraphics[width=\textwidth]{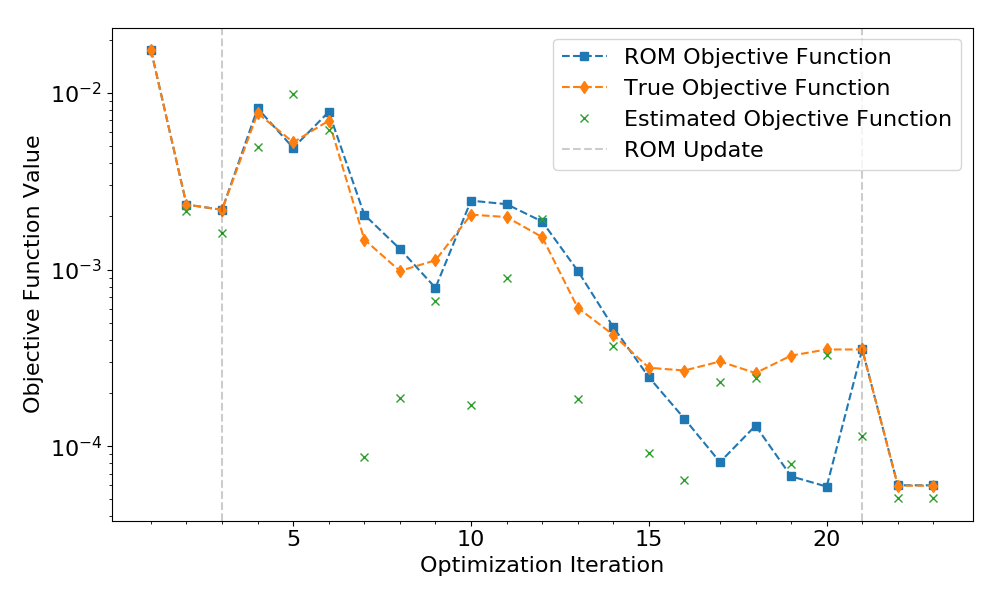}
\caption{Interpolatory Updates}
\label{fig:2D_convergence_updates}
\end{subfigure}
\caption{Convergence behavior of 3-point ROM without updates and 3-point ROM with Interpolatory Updates for a representative realization of the 2D experiment}
\label{fig:2D_convergence}
\end{figure}

\begin{figure}
    \centering
    \includegraphics[width=0.8\textwidth]{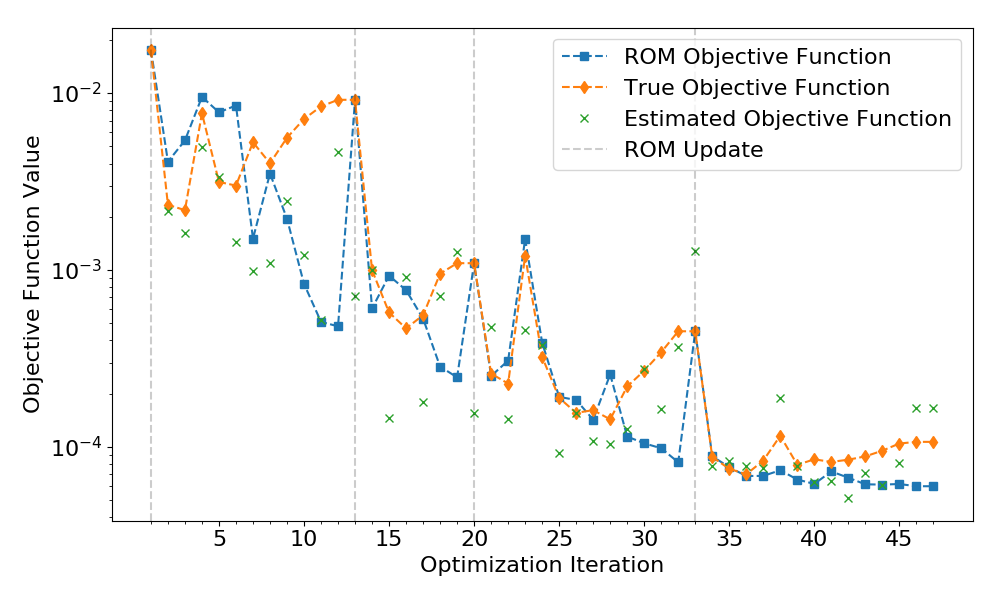}
    \caption{Convergence history of the 1-point ROM with Residual-Driven updates for a representative realization of the 2D experiment}
    \label{fig:2d_convergence_residual}
\end{figure}

\begin{figure}
    \centering
    \includegraphics[width=0.99\textwidth]{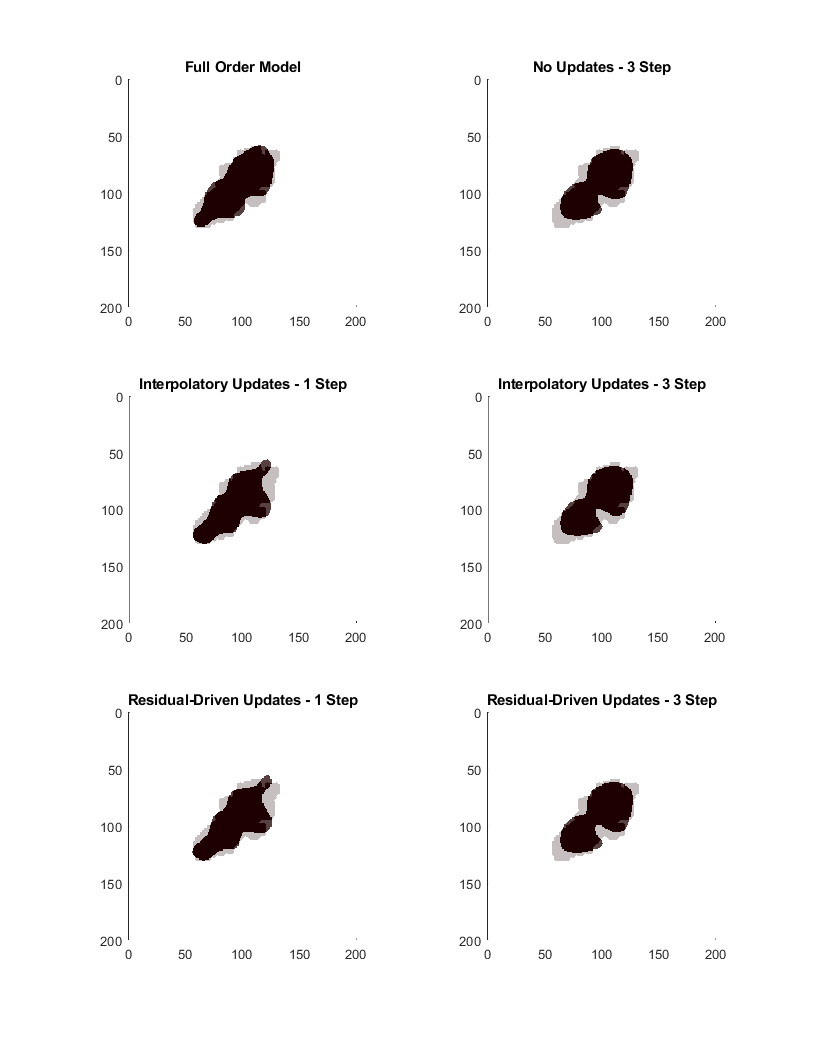}
    \caption{2D Reconstruction Results}
    \label{fig:2d_reconstruction}
\end{figure}

From \cref{table:2D_results}, we see that all the ROM-based methods result in a significant reduction in the number of large solves compared with the full order model. However, while the reconstruction using the 3-point ROM without updates (\cref{fig:2d_reconstruction}) is visually similar to that generated using the full-order model, \cref{fig:2D_convergence_none} shows that the (full-order) objective function is approximately 10 times larger than the 3-point reduced objective function at the final iteration. As seen in \cref{fig:2D_convergence_updates}, this discrepancy is detected by the probabilistic estimates at iteration 20, and the ROM is updated. For the remaining iterations, the ROM objective function closely matches the (full-order) objective function, and the optimizer terminates with the full-order objective function within the desired tolerance.

Thus, our error detection and update procedures are successful. However, the interpolatory update scheme incurs a non-trivial cost, approximately 45\% more solves in the median case or nearly 68\% more solves in the median amortized case. By using the residual-driven update scheme of \cref{subsec:Updates_Residual}, the costs are reduced to approximately 14\% more solves in the median case or approximately 21\% more solves in the median amortized case. We note that these extra costs do provide an increase in robustness.

Despite using a significantly smaller initial reduced-order model, the 1-point models are able to produce reconstructions of a similar quality to the 3-point models (as seen in \cref{fig:2d_reconstruction}) at the cost of multiple updates, but still with a significant reduction in number of large linear solves. A sample convergence history with three updates is seen in \cref{fig:2d_convergence_residual}. While using interpolatory update scheme with the 1-point models does require fewer large linear solves than the interpolatory update scheme used with the 3-point models, the residual-driven update scheme offers a significantly larger reduction compared with all other methods. The residual-driven update scheme required approximately 31\% \emph{fewer} large linear solutions in the median case and approximately 47\% \emph{fewer} solves in the median amortized case (that is, where the solutions from the initial conditions are re-used between optimizations).

\subsection{3D}\label{subsec:Experiements_3D}

\newcommand{\threedseed}{20}
\newcommand{\noiselevelthreed}{1}

The 3D problem is evaluated on a $64 \times 64 \times 64$ mesh, yielding $\numstate=\num{262144}$ degrees of freedom for the large system. The standard centered finite difference scheme is used to discretize \cref{eq:DOT_PDE1,eq:DOT_PDE2,eq:DOT_PDE3,eq:DOT_PDE4}. We use 961 sources and 961 detectors arranged as shown in \cref{fig:geometries}. The true image (before adding the random inhomogeneities) is shown in \cref{fig:3d_truth}. Note that this image is not exactly reconstructable using our choice of parameterization. We use 27 compactly-supported radial basis functions (CSRBF) for our PaLS reconstruction, yielding 135 (five per CSRBF) parameters. As in \cref{subsec:Experiements_2D}, the absorption values, $\mu_{out}$ and $\mu_{in}$, are set to 0.005 and 0.15, respectively. White noise is added to the measurements at a level of {\noiselevelthreed} permille. To quantify the impact of the random nature of the trace estimates of \cref{sec:Propabilistic_Estimates}, each experiment is repeated 15 times and the minimum, 25th, 50th, 75th percentiles, and maximum number of large linear solves are reported in \cref{table:3D_results}. \Cref{fig:3d_reconstruction} shows the reconstructions for a representative noise realization\footnote{Note that the anomaly is correctly localized in each case. Using additional frequencies may resolve the shape more finely.}.

\begin{figure}[ht]
\centering
\begin{subfigure}[b]{\textwidth}
\includegraphics[width=\textwidth]{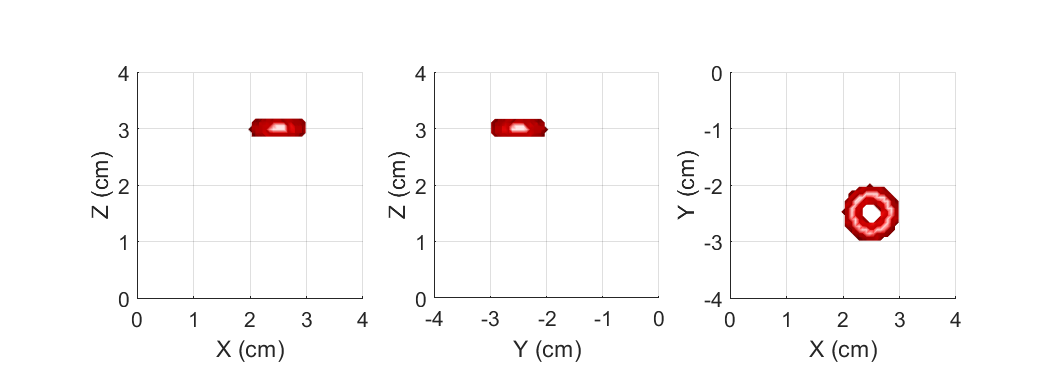}
\caption{Truth Image}
\end{subfigure}
\begin{subfigure}[b]{\textwidth}
\includegraphics[width=\textwidth]{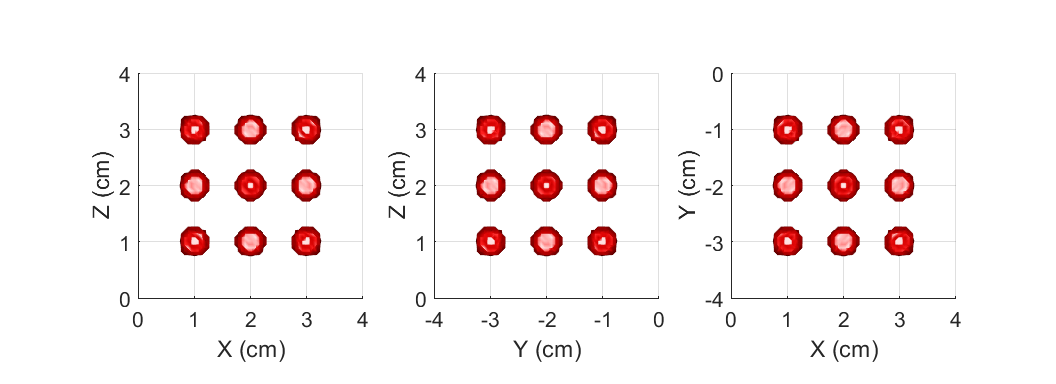}
\caption{Initial Condition}
\end{subfigure}
\caption{True Image for the 3D Model and Initial Condition}
\label{fig:3d_truth}
\end{figure}

\subsubsection{Results}

\begin{table}[ht]%
\centering
\begin{tabular}{|l|c|c|c|c|c|}%
\hline%
\multirow{2}{*}{Experiment Setup}&Median Solves&\multicolumn{4}{c|}{Percentiles}\\%
\cline{3%
-%
6}%
&{[}Amortized{]}&Min&25&75&Max\\%
\hline%
\multirow{2}{*}{Full Order Model}&23064&16337&22103&24025&26908\\%
&{[}21142{]}&{[}14415{]}&{[}20181{]}&{[}22103{]}&{[}24986{]}\\%
\hline%
\multirow{2}{*}{ROM - 3 Point - Interpolatory Updates at $\V{x}_\mathrm{p}$}&5776&5774&5775&5776&7704\\%
&{[}3854{]}&{[}3852{]}&{[}3853{]}&{[}3854{]}&{[}5782{]}\\%
\hline%
\multirow{2}{*}{ROM - 3 Point - No Updates}&5766&5766&5766&5766&5766\\%
&{[}3844{]}&{[}3844{]}&{[}3844{]}&{[}3844{]}&{[}3844{]}\\%
\hline%
\multirow{2}{*}{ROM - 3 Point - Residual-Driven Updates at $\V{x}_\mathrm{p}$}&5776&5774&5775&5777&5797\\%
&{[}3854{]}&{[}3852{]}&{[}3853{]}&{[}3855{]}&{[}3875{]}\\%
\hline%
\multirow{2}{*}{ROM - 1 Point - Interpolatory Updates at $\V{x}_\mathrm{p}$}&3871&3858&3867&3881&5798\\%
&{[}1949{]}&{[}1936{]}&{[}1945{]}&{[}1959{]}&{[}3876{]}\\%
\hline%
\multirow{2}{*}{ROM - 1 Point - Residual-Driven Updates at $\V{x}_\mathrm{p}$}&1985&1953&1965&1993&2006\\%
&{[}63{]}&{[}31{]}&{[}43{]}&{[}71{]}&{[}84{]}\\%
\hline%
\end{tabular}%
\caption{Number of Large Linear Solves for the 3D Problem}%
\label{table:3D_results}%
\end{table}

\begin{figure}
    \centering
    \includegraphics[angle=-90,width=\textwidth]{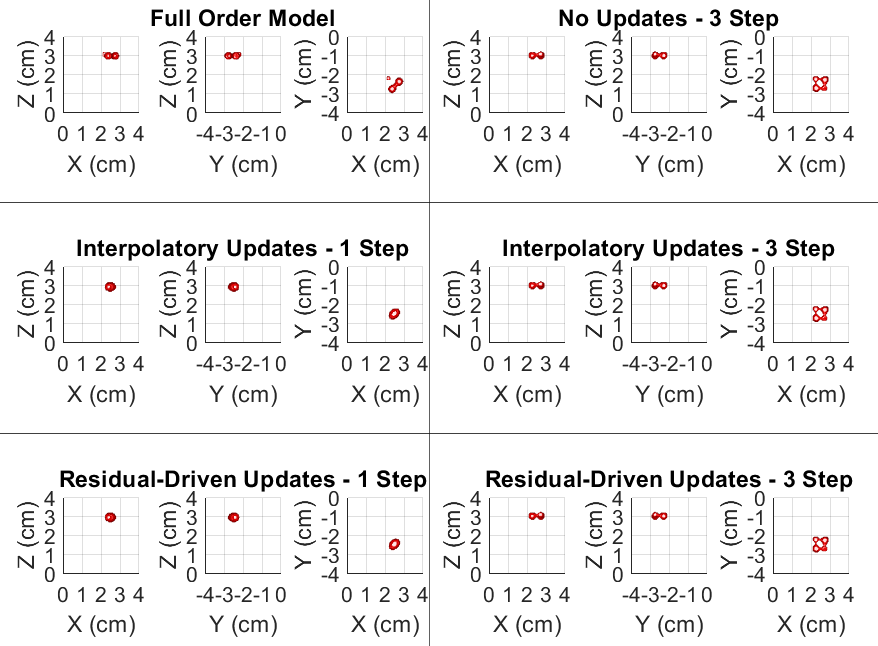}
    \caption{3D Reconstruction Results}
    \label{fig:3d_reconstruction}
\end{figure}

The 3D results are very similar to the 2D results discussed in \cref{subsec:Experiements_2D}, but the improvements are even more striking. The 3-point models do not require updates in most cases, resulting in a small increase in the number of large linear solutions required due to the objective function estimation. However, there are significant savings when using the 1-point reduced-order model. Even with interpolatory updates, there is a reduction of approximately 32\% in the median case and a reduction of approximately 49\% in the median amortized case. Using residual-driven updates results in a reduction of approximately 66\% in the median case and a reduction of approximately 98\% in the median amortized case. Thus, in the amortized case, our proposed technique requires a factor 50 fewer large linear solutions than the 3-point reduced-order model and results in a more than 300-fold reduction compared with the full-order model.

\section{Conclusions and Future Work}\label{sec:Conclusions}

We have shown that the use of stochastic trace estimation enables the estimation of reduced order model quality for the DOT problem. This allows us to adaptively update the model during the optimization. With a modest cost increase, a larger initial reduced-order model can use these estimates for robustness. 
We have also presented an update scheme for projection-based model reduction that is able to significantly reduce the relative error in the objective function at the cost of just a handful of large linear solutions compared with an interpolatory update ($O(10)$ vs $O(\numobservable+\numcontrol)$). By starting with a lower order initial reduced-order model, using stochastic estimates to guide adaptivity, and updating using our residual-driven scheme, we obtain a method that is very efficient and robust. 

The use of stochastic estimates and updates allows the user to choose a balance of efficiency and robustness. For the purposes of this paper, we choose to focus on efficiency, but this choice may vary on a per-application basis.
To reduce the cost of sampling, the stochastic estimates could be performed only at accepted (by the trust-region) steps instead at every trial point. This may require more accurate estimates and hence requires balancing between fewer sampling locations and more samples per location. 
In \cref{subsec:Experiements_2D}, experimental evidence suggests updating at $\params_\mathrm{p}$ was usually more efficient than updating at $\params_\mathrm{c}$, but further analysis is necessary. Furthermore, extending the projection bases using solutions from both points could reduce the number of updates required at the cost of more expensive updates.
In our work, we use the same number of samples for each objective function estimate. Robustness could be improved by increasing the number of samples taken (and thus reducing the variance of the estimate) as the optimizer approaches convergence.

As discussed in \cref{subsec:Background_TR} and \cref{sec:Propabilistic_Estimates}, we believe that these estimates combined with the theory developed in \cite{Beattie2012} can be used to show convergence with high probability in an extension of the (trust region) conditional model framework. This will be explored in future work along with using the bounds of \cref{sec:Propabilistic_Estimates} to extend the ``error-aware'' trust-region method in \cite{Yue2013}.


\section{Acknowledgements}
We thank Misha Kilmer for the use of the PaLS code \cite{Aghasi2011a}
and Serkan Gugercin and Chris Beattie for many helpful discussions.

\bibliographystyle{siamplain}
\bibliography{library_RobParmInv}
\end{document}